\documentclass[12pt]{article}

\usepackage{amsmath,amssymb,amscd,amsthm,amsfonts}
\usepackage{graphicx}
\usepackage{hyperref}
\usepackage{dsfont}
\usepackage{xcolor}

\newtheorem{theorem}{Theorem}[section]
\newtheorem*{theoremp}{Theorem}
\newtheorem{lemma}[theorem]{Lemma}

\newtheorem{corollary}[theorem]{Corollary}

\newtheorem*{problem}{Problem}
\newtheorem{definition}[theorem]{Definition}

\numberwithin{figure}{section}

\makeatletter
\newtheorem*{rep@theorem}{\rep@title}
\newcommand{\newreptheorem}[2]{
\newenvironment{rep#1}[1]{
 \def\rep@title{#2 \ref{##1}}
 \begin{rep@theorem}}
 {\end{rep@theorem}}}
\makeatother

\newreptheorem{theorem}{Theorem}

\newcommand{\rr}{\mathds{R}}

\DeclareMathOperator{\conv}{conv}
\DeclareMathOperator{\vol}{Vol}
\DeclareMathOperator{\diam}{diam}

\newcommand{\ff}{\mathcal{F}}

\newcounter{examplecounter}
\newcommand{\example}[1]{
	\stepcounter{examplecounter}{
		\noindent\textbf{Example~\theexamplecounter.}
		{ #1 \par}
		\bigskip
	}{}
}

\title{Positive-fraction intersection results and variations of weak epsilon-nets}

\author{Alexander Magazinov\thanks{Supported by ERC Advanced Research Grant no. 267165 (DISCONV).} \and Pablo Sober\'on}

\begin{document}

\maketitle

\begin{abstract}
Given a finite set $X$ of points in $R^n$ and a family $F$ of sets generated by the pairs of points of $X$, we determine volumetric and structural conditions for the sets that allow us to guarantee the existence of a positive-fraction subfamily $F'$ of $F$ for which the sets have non-empty intersection.  This allows us to show the existence of weak epsilon-nets for these families.  We also prove a topological variation of the existence of weak epsilon-nets for convex sets.
\end{abstract}



\section{Introduction}

The results of this paper were motivated by the following problem by Imre  B\'ar\'any.

\begin{problem}
For any two points $x,y \in \rr^n$, let $S(x,y)$ be the Euclidean ball with diameter $[x,y]$.  Find the optimal constant $b_n$, if it exists, depending only on the dimension such that for any set $X \subset \rr^n$ with $|X|= N$, there a point $z$ contained in at least $b_n N^2$ sets of the form $S(x,y)$ with $x, y \in X$.
\end{problem}

This is representative of a wide class of problems in discrete geometry which we call {\it positive-fraction problems}. A {\it positive-fraction problem}
is a problem that has the following setting. Suppose we are given a family $\mathcal{A}$ of sets in $\rr^n$, and assume that
the family $\mathcal A$ has some prescribed property (e.g., all members of $\mathcal A$ are convex, or $\mathcal A$
is the family of all $n$-simplices on given vertices, etc.). Determine if there exists $\alpha > 0$ depending only on the dimension $n$
and on the property of $\mathcal A$, but not on $|\mathcal{A}|$, such that there is a subfamily $\mathcal B \subset \mathcal A$
with non-empty intersection and $|\mathcal B| \geq \alpha |\mathcal A|$.

There are several well-known positive-fraction problems that have an affirmative answer (a {\it positive-fraction result}). Classic examples include the fractional Helly theorem by Katchalski and Liu \cite{Katchalski:1979bq} and the first selection lemma by B\'ar\'any \cite[Thm.~5.1]{Barany:1982va} 
(as called in \cite{Matousek:2002td}).

The first selection lemma states that for any finite set $X$ of points in $\rr^n$, there is a point $z$ contained in a positive fraction $c_n$ of the simplices spanned by $X$, where $c_n$ depends on the dimension.

In 1984 Boros and F\"uredi \cite{Boros:1984ba} proved that $c_2 \ge \frac29$.  Different proofs of this result have been discovered 
recently \cite{Bukh:2006wm, Fox:ks}. The cases $n=1,2$ are the only where the optimal constant is known \cite{Bukh:2010hz}.  The first high-dimensional version was proved by B\'ar\'any \cite{Barany:1982va}.  Namely,

\begin{theoremp}
[First selection lemma]
Let $S \subset \rr^n$, with $|S|=N$.  There is a constant $c_n$ depending only on the dimension such that we can always find an intersecting family $\ff$ of simplices with vertices in $S$ such that
\[
|\ff| \ge c_n \binom{N}{n+1}
\]	
\end{theoremp}

There have been several improvements on the result above; either finding better bounds on the constant $c_n$ \cite{Gromov:2010eb, MW11, KMS12}, or requiring more conditions that the intersecting simplices have to satisfy, as in \cite{Pach:1980un, Karasev:2012bj, Jia14, soberon2015}.

While an $n$-dimensional simplex is a natural hull of $n + 1$ points, there are several ways to define a hull of two points in
$\rr^n$. One is, as in B\'ar\'any's problem, to consider Euclidean ball with the diameter
$[x, y]$ as a hull of two points $x, y \in \rr^n$. For this setup we prove a bound $b_n \geq \frac{1}{n + 1}$ in Theorem \ref{theorem-many-balls-intersect}.

One can notice that the first selection lemma itself gives a positive answer for B\'ar\'any's problem. Indeed, take $z$ to be
a point that is contained in a positive fraction of simplices with vertices in $X$. One can show that for each simplex
$\conv \{ x_1, x_2, \ldots, x_{n + 1} \} \ni z$ one has $z \in S(x_i, x_j)$ for at least $\frac{n + 1}{2}$ pairs $(i, j)$.
If we record all such pairs of points from all simplices containing $z$, then a pair of points cannot be mentioned more than
$\binom{|X|}{n - 1}$ times, so sufficiently many distinct pairs are mentioned. The arguments needed to show that $z$ is covered many times are similar to the ones used in~\cite{Barany:1987ki}.

However, the estimate for $b_n$ obtained by the double counting method above is only about $1 / (n!)$. At the moment the authors
are not aware of any shape $S(x, y)$ for which the double counting method with the first selection lemma gives the best constants. Therefore we leave its details to the interested reader.

It is interesting to consider ``thinner'' shapes as the hull of two points (i.e. the shapes that are more stretched along
the segment $[x, y]$ than a ball).  For this purpose we introduce the following definition of a $t$-shape (with $t \in (0, 1)$).

\begin{definition}\label{definition-t-shape}
Let $t > 0$.  A mapping $S$ from $\rr^n \times \rr^n$ to the set of measurable subsets of $\rr^n$ will be called a $t$-shape if for all $x\neq y$ one has $S(x,y) = S(y,x)$ and for all $r$ that satisfy $|x-y|\ge r>0$ one has
\begin{eqnarray*}
	\vol(S(x,y)\cap B_r(y)) & \ge & t\cdot \vol (B_r(y)),
\end{eqnarray*}
where $\vol({\cdot})$ stands for the Euclidean volume, $B_r(x)$ is the closed ball of radius $r$ around $x$, $|x-y|$ is the Euclidean 
distance between $x$ and $y$.
\end{definition}

This definition is very relaxed, the following more familiar shapes are examples of $t$-shapes for some $t$.

\example{
For every $a>0$, the ellipsoids
\[
E_a(x,y) = \{z \in \rr^n : |z-x| + |z-y| \le (1+a) |x-y| \}
\]
are $t$-shapes for some $t=t_1(n,a)$.
}

\example{
For every $\pi>a>0$, the shapes
\[
S_a(x,y) = \{z \in \rr^n : \angle(x-z,y-z) \ge a  \} 
\]
are $t$-shapes for some $t=t_2(n,a)$.  In particular, if $a = \pi / 2$, then $S_a(x,y)$ is simply the Euclidean ball with diameter $[x,y]$.  
For a general $a$ these sets are also called \textit{$a$-lenses}, and families of these objects have nice intersection properties~%
\cite{Barany:1987fi, Barany:1987ef}.
}

We give a positive answer to B\'ar\'any's problem where $S(x, y)$ are no longer required to be balls, but are $t$-shapes for some $t > 0$.
In this case the fraction of intersecting hulls we can guarantee
decreases exponentially in $n$ for every fixed $t$ .

\begin{theorem}[Positive fraction intersection for $t$-shapes]\label{theorem-positive-thin-shapes}
There exist positive absolute constants $c_1$ and $c_2$ satisfying the following property.
For every $t$-shape $S(x,y)$ in $\rr^n$, and for every finite $X$ with $|X| \ge N$ setting 
$$\lambda = c_1 \cdot t \cdot c_2^n$$
guarantees that there is a point $z \in \rr^n$ that is covered by at least $\lambda N^2$ shapes 
$S(x,y)$ with $x,y \in X$.
\end{theorem}

In Theorem \ref{theorem-many-boxes-intersect} we also consider the shape $S(x, y)$ to be the minimal box (i.e. a parallelotope with edges parallel to coordinate lines)
containing $x$ and $y$. A box is not a $t$-shape for any $t > 0$, but a positive-fraction result can also be proved. In this case
the fraction also decreases exponentially in the dimension.

We should emphasise that a condition on the sets $S(x,y)$ is necessary, as considering $S(x,y)$ to be the segment $[x,y]$ fails to give a positive-fraction result.

One of the most striking applications of the first selection lemma is the proof of existence of weak $\varepsilon$-nets for convex sets, presented below.

\begin{theoremp}[Alon, B\'ar\'any, F\"uredi and Kleitman, 1992 \cite{Alon:2008ek}]\label{theorem-weak-epison}
Let $n$ be a positive integer, and $1 \ge \varepsilon > 0$.  Then, there is a positive integer $m = m(n,\varepsilon)$ such that the following holds.  For every finite set $S \subset \rr^n$, there is a set $\mathcal{T} \subset \rr^n$ of $m$ points such that if $A \subset S$ is a subset with size at least $\varepsilon |S|$, then
\[
\mathcal{T} \cap \operatorname{conv}(A) \neq \emptyset.
\]
Moreover, $m = O(\varepsilon^{-n-1})$ where the implied constant of the $O$ notation depends on $n$.
\end{theoremp}

The set $\mathcal{T}$ is called a weak $\varepsilon$-net of $S$.  Bounding the size of a weak $\varepsilon$-net for convex sets is a notorious problem.  The best improvement over the bound above is $m = O(\varepsilon^{-n}\cdot \operatorname{polylog}(\varepsilon^{-1}))$ \cite{Chazelle:1993ft, Matouaek:2003gd}.  The best lower bound is 
$m = \Omega (\varepsilon^{-1}\cdot \log^{n-1} (\varepsilon^{-1}))$ \cite{Bukh:2011vs}.

We explore variations of weak $\varepsilon$-nets for operators different from the convex hull.  For instance, in Theorem \ref{theorem-weakepsilon-topological} we show that the topological versions of the selection theorem imply directly a topological extension of weak $\varepsilon$-nets for convex sets, using the same arguments as \cite{Alon:2008ek}.  This generalises weak $\varepsilon$-nets just like the topological Tverberg theorem generalises Tverberg's theorem 
\cite{Barany:1981vh, Tverberg:1966tb}. We also consider variants of weak $\varepsilon$-nets for $t$-shapes.  Given a $t$-shape $S$, we can define the \textit{thin hull} of a set $A$ as
\[
\operatorname{thin}_S(A) = \bigcup \{S(x,y) : x,y \in A\}
\]

Since $t$-shapes admit a first selection lemma, this operator begs for the existence of weak $\varepsilon$-nets.

\begin{theorem}\label{theorem-thin-hull-epsilon-net}
	Let $n$ be a positive integer, $\varepsilon, t >0$, and $S$ a $t$-shape in $\rr^n$.  Then, there is a positive integer $m' = m'(t,\varepsilon, n)$ such that the following holds.  For any finite set $X \subset \rr^n$, there is a set $\mathcal{T} \subset \rr^n$ of size $m'$ such that if $A \subset X$ and $|A| \ge \varepsilon |X|$, then 
	\[
	T \cap \operatorname{thin}_S (A) \neq \emptyset.
	\]
	Moreover, $m' = O(\varepsilon^{-2})$, where the implied constant of the $O$ notation depends on $n$ and $t$.
\end{theorem}

The rest of the paper is organised as follows.  In Section \ref{section-balls} we give a solution B\'ar\'any's original problem.  In Section \ref{section-t-shapes} we prove our results on $t$-shapes.  In Section \ref{section-boxes} we state and prove our results for boxes.  Finally, in Section \ref{section-epsilonnets} we prove all our results regarding weak $\varepsilon$-nets.

\section{Positive-fraction result for Euclidean balls}\label{section-balls}

In this section we give a solution to the original B\'ar\'any's problem with a large constant.  We prove a stronger version of the result, in the same spirit as Karasev's colourful version of B\'ar\'any's result \cite{Karasev:2012bj}.  Namely, instead of having one finite set $X$, we are given two sets $A, B$.  We give a positive fraction intersection result for the balls having diameters with one end in $A$ and another in $B$.  The case $A=B$ is B\'ar\'any's problem.

\begin{theorem}\label{theorem-many-balls-intersect}
For each $x,y \in \rr^n$ let $S(x,y)$ be the Euclidean ball with diameter $[x,y]$.  Then, for finite sets $A, B \subset \rr^n$ of $N$ and $M$ points respectively, there is a point covered by at least $\frac{1}{n + 1}NM$ sets of the form $S(a,b)$ with $a \in A,b \in B$.
\end{theorem}

\begin{proof}
By Rado's central point theorem \cite{Rado:1946lms}, there is a point $z \in \mathbb R^n$ such that for every close half-space that contains $z$ also contains at least $\frac{N}{n + 1}$ points of the set $A$.  Consider a point $b \in B$. If $b = z$, then the ball $S(a, b)$ contains $z$ for any $a \in A$.

If $b \neq z$, then consider a hyperplane through $z$ orthogonal to the segment $[b, z]$. Let $H$ be the closed half-space of this hyperplane
that does not contain $b$. Then for any $a \in A \cap H$ the angle $\angle bza$ is not acute, and therefore $z \in S(a, b)$.
Hence every point $b \in B$ is contained in at least $\frac{N}{n + 1}$ pairs $(a, b)$ with $a \in A$ such that $z \in S(a, b)$.

This gives a total of at least $\frac{NM}{(n + 1)}$ ordered pairs $a \in A, b \in B$ such that $z \in S(a, b)$, as desired.

\end{proof}


\section{Positive-fraction results for \texorpdfstring{$t$}{t}-shapes}\label{section-t-shapes}

Before proving Theorem \ref{theorem-positive-thin-shapes}, notice that it can be naturally extended to measures in $\rr^n$ by usual approximation arguments.

\begin{theorem}\label{theorem-thin-shapes-probabilistic}
There exist positive absolute constants ${c_1}$ and $c_2$ satisfying the following property.
For every $t$-shape $S(x,y)$ in $\rr^n$ and for every Borel probability measure $\mu$ in $\rr^n$, setting 
$$\lambda = c_1 \cdot t \cdot c_2^n$$
guarantees that there is a point $z \in \rr^n$ such that
	 \[
	 \mathbb{P}(z \in S(x,y): x,y \ \mbox{are independent } \mu\mbox{-random points}) \ge \lambda.
	 \]
\end{theorem}

We now prove Theorem~\ref{theorem-positive-thin-shapes} (and thus Theorem~\ref{theorem-thin-shapes-probabilistic} as well).  The first ingredient we need is the following lemma.

\begin{lemma}\label{lemma-diameters}
Let $(X, \rho)$ be a finite metric spaces with $|X| = N > 1$.  Then, there is a subset $Y \subset X$ such that one can find at least 
$N^2/64$ ordered pairs $(y,x)$ with $y \in Y$, $x \in X \setminus Y$ and $\rho(x,y) \ge \frac1{4} \diam{Y}$.
\end{lemma}

\begin{proof}
Let $B_r(x)$ be the closed $\rho$-ball in $X$ of radius $r$ centered at $x$. For each $x \in X$ let
\[
a(x) = \max\left\{r \in \rr_+ : | B_r(x)| \le \frac{3N}4\right\}.
\]
	We can enumerate the points of $X$ as $x_1, x_2, \ldots, x_N$ so that
\[
a(x_1) \le a(x_2) \le \ldots \le a(x_N).	
\]

We will prove that the set $Y = \left\{x_i : i \le \left\lfloor \frac{3N}4 \right\rfloor\right\}$ satisfies the conclusion of the lemma.

Set $a = a (x_{ \left\lfloor {3N}/4 \right\rfloor})$.  First we show that $\diam{Y} \le 2a$.  Indeed, for every two points $y_1, y_2 \in Y$ the sets 
$$B_{a(y_1)}(y_1) \cap X \quad \text{and} \quad B_{a(y_2)}(y_2) \cap X$$
are each of cardinality at least $\frac{3N}4$ and thus have a point of intersection $z$.  Then
\[
\rho (y_1, y_2) \le \rho (y_1,z) + \rho (z,y_2) \le a(y_1) + a(y_2) \le 2a.
\]
Now consider the graph $G=(V,E)$ with vertex set $V=X \setminus Y$ such that two points $v_1, v_2 \in V$ are connected if and only if $\rho( v_1, v_2) \geq a$.  Let $\alpha (G)$ be the independence number of $G$. 
We consider two cases that cover all possibilities.

\textbf{Case 1.} $\alpha(G) \le \frac{N}8$.  We will show that for every $y \in Y$ there are at least 
$N/8$ pairs $(y,v)$ such that $v \in V$ and $\rho(y,v) \geq a/2$.  

Assume the opposite: for some $y \in Y$ there are fewer than $N/8$ pairs $(y,v)$ with $v \in V$ such that $\rho(y,v) \geq a/2$.
From the assumption it follows that for more than $|V| - N/8$ points $v\in V$ the inequality $\rho(y,v) < a/2$ holds. 
We can choose $\varepsilon > 0$ to be small enough to satisfy the following condition: for every $v \in V$ one has 
$\rho(y,v) \leq a/2 - \varepsilon$ whenever $\rho(y,v) < a/2$. Hence $|V \cap B_{a/2 - \varepsilon}(y)| > |V| - N/8$.

On the other hand,
\[
|V| = N - |Y| = N- \left\lfloor \frac{3N}{4}\right\rfloor \ge \frac{N}4.
\]
Hence $|V \cap B_{a/2 - \varepsilon}(y)| > \frac{N}4 - \frac{N}8 = \frac{N}8$. By definition of Case 1, the set
$V \cap B_{a/2 - \varepsilon}(y)$ is too large to be independent in $G$.
But $\diam B_{a/2 - \varepsilon}(y) < a$, and therefore
the set $V \cap B_{a/2 - \varepsilon}(y)$ has to be independent in $G$, a contradiction.

So, there are at least $\frac{N}8 \cdot \left\lfloor \frac{3N}{4} \right\rfloor$ pairs $(y,v)$ such that $y \in Y, v\in V$ and $\rho (y,v) > a/2$, which is sufficient for the lemma.

\textbf{Case 2.} $\alpha(G) > \frac{N}8$.
Let $W \subset V$ be an independent set in $G$ with $|W| > \frac{N}8$. 

First notice that, by the choice of $Y$ and $V$, each point $v \in V$ satisfies
\[
|X \setminus B_{a-\epsilon}(v)| \ge \lceil N/4 \rceil
\]
for every $\varepsilon \in (0, a)$.

Fix a point $w \in W$. Choose $\varepsilon \in (0, a)$ to be small enough to satisfy the following condition: for every $x \in X$ one has 
$\rho(w, x) \leq a - \varepsilon$ whenever $\rho(w, x) < a$.

For every $w' \in W$ ($w' \neq w$), the vertices $w$ and $w'$ are not connected by an edge of $G$. Thus $\rho(w, w') < a$, and,
consequently,  $\rho(w, w') \leq a - \varepsilon$. Hence $W \subset B_{a - \varepsilon}(w)$.

Therefore 
$$X \setminus B_{a - \varepsilon}(w) = (V\setminus B_{a - \varepsilon}(w)) \cup (Y \setminus B_{a - \varepsilon}(w)) \subset
(V\setminus W) \cup (Y \setminus B_{a - \varepsilon}(w)).$$

As a result, we have
$$|X \setminus B_{a - \varepsilon}(w)| \leq |Y \setminus B_{a - \varepsilon}(w)| + |V\setminus W| =
|Y \setminus B_{a - \varepsilon}(w)| + |V| - |W|,$$
and so
$$|Y \setminus B_{a - \varepsilon}(w)| \geq |X \setminus B_{a - \varepsilon}(w)| + |W| - |V| \geq 
\lceil N/4 \rceil + N/8 - \lceil N/4 \rceil = N/8.$$

Thus, every $w \in W$ produces at least $\frac{N}{8}$ pairs $(y,w)$ with $y \in Y$ and $\rho (y,w) > a - \varepsilon$.
By the choice of $\varepsilon$, all such $y$ satisfy $\rho (y,w) \geq a$ as well. Iterating over all $w \in W$, we get
a total of at least $\frac{N}{8} \cdot \frac{N}{8} = \frac{N^2}{64}$ pairs $(y, x)$ (where $x = w$) for the lemma.

\end{proof}

We are now ready to prove Theorem~\ref{theorem-positive-thin-shapes}.

\begin{proof}[Proof of Theorem~\ref{theorem-positive-thin-shapes}]
Let $U_n$ be the volume of the unit ball in $\rr^n$.

For the set $X \subset \rr^n$, choose a subset $Y$ which satisfies the conditions of Lemma \ref{lemma-diameters}.  Let $d = \diam{Y}$.

Define
$$R = \bigcup_{y \in Y} B_d(y).$$

Note that $Y \subset B_d(y)$ for any $y \in Y$. As a consequence, $R \subset B_{2d}(y)$. Thus, $\vol (R) \leq (2d)^n U_n$.

On the other hand, for all pairs $(y,x)$ with $y \in Y$, $x \in X \setminus Y$ and $|x-y| > \frac{d}{4}$ one has
\[
\vol{(S(x,y) \cap R)} \ge t \cdot \frac{d^n}{4^n} \cdot U_n.
\]

Since the number of such pairs is at least $\frac{N^2}{64}$, there is a point in $R$ that is covered by at least
\[
\frac{t*(d^n/4^n)*U_n}{\vol (R)}\cdot \frac{N^2}{64} \ge \frac{t \cdot N^2}{2^n \cdot 4^{n+3}}
\]
sets of the form $S(x,y)$. Hence Theorem~\ref{theorem-positive-thin-shapes} is proved with $c_1 = 1/64$, $c_2 = 1/8$.

\end{proof}

\section{Positive-fraction result for boxes}\label{section-boxes}

A {\it box} is a (closed) parallelotope with all edges parallel to coordinate axes. For arbitrary $t > 0$ a box is
not a $t$-shape, because it can be arbitrarily flat along any coordinate plane. Nevertheless, we prove a
positive-fraction result for boxes.

\begin{theorem}\label{theorem-many-boxes-intersect}
For each $x,y \in \rr^n$ let $S(x,y)$ be the minimal box that contains $x$ and $y$. Then,
\begin{itemize}
	\item for any finite set $X \subset \rr^n$ of $N$ 
points, there is a point covered by at least $\frac{N}{2^n}\left( \frac{N}{2^n} - 1 \right)$ sets of the form $S(x,y)$ with $x,y \in X$, and
	\item for each $\varepsilon > 0$, there is a finite set $X \subset \rr^n$ of $N$ points such that no point is contained in more than $\left(\frac1{2^n}+\varepsilon\right)N^2$ sets of the form $S(x,y)$ with $x,y \in X$.
\end{itemize}
\end{theorem}

For a point $x \in \rr^n$, we denote by $(x_1, x_2, \ldots, x_n)$ its coordinates.  We prove Theorem~\ref{theorem-many-boxes-intersect} via the following lemma.

\begin{lemma}\label{lemma-boxes}
Let $X\subset \rr^n$ be a set of $N$ points, $j \in \{ 1, 2, \ldots, n \}$. Then there exist two sequences
\[ z_1, z_2, \ldots, z_j \; (z_i \in \rr) \quad \text{and} \quad
\varepsilon_1, \varepsilon_2, \ldots, \varepsilon_j \; (\varepsilon_i \in \{ -1, +1 \})\]
such that the two sets
\begin{eqnarray*}
R_1 & = & \{ x \in \rr^n : \varepsilon_i(x_i - z_i) \leq 0 \; \text{for} \; i = 1, 2, \ldots, j \}, \\
R_2 & = & \{ x \in \rr^n : \varepsilon_i(x_i - z_i) \geq 0 \; \text{for} \; i = 1, 2, \ldots, j \}
\end{eqnarray*}
satisfy $|X \cap R_k| \geq \frac{|X|}{2^j}$ for $k = 1,2$. 
\end{lemma}

\begin{proof}
We use induction over $j$.  Let $j = 1$ and choose $z_1$ so that the plane $x_1 = z_1$ splits the set $X$ into almost equal parts (this, by definition,
means that the intersection of $X$ with each of two closed subspaces $x_1 \leq z_1$ and $x_1 \geq z_1$ contains at least $|X| / 2$
points). It is clear that $\varepsilon_1 = +1$ will suffice.

Let $j > 1$. By induction hypothesis, we can find $z_1, z_2, \ldots, z_{j - 1}$ and $\varepsilon_1, \varepsilon_2, \ldots, \varepsilon_{j - 1}$
that satisfy the statement for $j - 1$. Set
$$X_1^{(j - 1)} = X \cap \{ x \in \rr^n : \varepsilon_i(x_i - z_i) \leq 0 \; \text{for} \; i = 1, 2, \ldots, j - 1 \},$$
$$X_2^{(j - 1)} = X \cap \{ x \in \rr^n : \varepsilon_i(x_i - z_i) \geq 0 \; \text{for} \; i = 1, 2, \ldots, j - 1 \}.$$
The inductive assumption implies that
$$|X_k^{(j - 1)}| \geq \frac{|X|}{2^{j - 1}} \quad \text{for $k = 1, 2$}.$$

For $k = 1, 2$ choose $z_j^{(k)}$ so that the plane $x_j = z_j^{(k)}$ splits the set $X_k^{(j - 1)}$ into almost equal parts.
Now set $z_j = \frac{z_j^{(1)} + z_j^{(2)}}{2}$, and $\varepsilon_j = +1$ if $z_j^{(1)} < z_j^{(2)}$, or 
$\varepsilon_j = -1$ otherwise. This is sufficient to complete the induction step.

\end{proof}

\begin{proof}[Proof of Theorem~\ref{theorem-many-boxes-intersect}]
Apply Lemma~\ref{lemma-boxes} to the set $X$ with $j = n$. Set $z = (z_1, z_2, \ldots, z_n)$, $X_1 = X \cap R_1$, 
$X_2 = (X \cap R_2) \setminus \{ z \}$. For every $x_1 \in X_1$ and every $x_2 \in X_2$ the box $S(x_1, x_2)$ contains $z$.

Further, $|X_1| \geq \frac{N}{2^n}$, $|X_2| \geq \frac{N}{2^n} - 1$ and $X_1 \cap X_2 = \emptyset$, which gives the necessary
number of pairs.

For the lower bound, notice that by standard approximation arguments (see, for instance, \cite{Mat03}) it suffices to find a probability measure $\mu$ in $\rr^n$ which is absolutely continuous with respect to the Lebesgue measure such that for all $z \in \rr^n$, the probability that $z \in S(x,y)$ is at most $\frac{1}{2^n}$ where $x,y$ are independent $\mu$-random points.  Let $\mu$ be the uniform probability measure on the unit cube $C=[0,1]^n$.  


Given $z=(z_1, z_2, \ldots, z_n)$ and $x,y \in \rr^n$, notice that $z \in S(x,y)$ if and only if there is a sequence $a=(\varepsilon_1, \varepsilon_2, \ldots, \varepsilon_n)\in \{+1,-1\}^n$ such that
\begin{eqnarray*}
	x \in R^{a}_1 & = & \{ u \in \rr^n : \varepsilon_i(u_i - z_i) \leq 0 \; \text{for} \; i = 1, 2, \ldots, n \}, \\
	y \in R^{a}_2 & = & \{ u \in \rr^n : \varepsilon_i(u_i - z_i) \geq 0 \; \text{for} \; i = 1, 2, \ldots, n \}
\end{eqnarray*}

Thus
\[
\mathbb{P}[z \in S(x,y): x,y \ \mbox{i.i.d.}] = \sum_{a \in \{+1,-1\}^n} \vol(R^{a}_1 \cap C) \cdot \vol (R^a_2 \cap C)
\]

Also, if $z \in S(x,y)$ for any $x,y$, it is necessary that $0 \le z_i \le 1$.  In order to bound $\vol(R^{a}_1 \cap C) \cdot \vol (R^a_2 \cap C)$, we may assume without loss of generality that $a = (+1,+1,\ldots, +1)$.  Then
\begin{eqnarray*}
	\vol(R^{a}_1 \cap C) \cdot \vol (R^a_2 \cap C)&  = & \left( \prod_{i=1}^n z_i \right) \left( \prod_{i=1}^n (1-z_i)\right) \\
	& = & \prod_{i=1}^n z_i (1-z_i) \le \prod_{i=1}^n \frac14 = \frac1{4^n}
\end{eqnarray*}

which gives us the desired conclusion.
\[
P[z \in S(x,y):  x,y \ \mbox{i.i.d.}] = \sum_{a \in \{+1,-1\}^n} \vol(R^{a}_1 \cap C) \cdot \vol (R^a_2 \cap C) \le \frac1{2^n}
\]

\end{proof}

\section{Results for weak $\varepsilon$-nets}\label{section-epsilonnets}

In this section we show that the notion of weak $\varepsilon$-nets can be extended to operators other than the convex hull.  The arguments we use are based on the original proof 
in \cite{Alon:2008ek}.

The first topological extension of the first selection lemma in arbitrary dimension, stated below, was obtain by Gromov \cite{Gromov:2010eb}, which extends to continuous maps.

\begin{theoremp}
Let $\Delta^{N-1}$ be the $(N-1)$-dimensional simplex and $f:\Delta^{N-1} \to \rr^n$ a continuous map.  There is a constant $c_{n}^*$ depending only on $n$ such that we can always find a family $\ff$ of $n$-dimensional faces of $\Delta^{N-1}$ such that the images of $\ff$ intersect and
\[
|\ff| \ge c_n^* \binom{N}{n+1}
\]
Moreover,
\[
c_n^* \ge \frac{2n}{(n+1)(n+1)!}
\]	
\end{theoremp}

There are now improved bounds on $c_n^*$ \cite{KMS12} (see also \cite{MW11}).  When $f$ is linear, we obtain the classic version of the first selection lemma.  A simplified proof of the result above is contained 
in~\cite{Karasev:2012bj}.  Using this result, one can prove a topological version of the weak $\varepsilon$-net result of \cite{Alon:2008ek} with an analogous proof.

\begin{theorem} \label{theorem-weakepsilon-topological}
	Let $n$ be a positive integer, $\varepsilon >0$.  Then, there is a positive integer $m_{\operatorname{top}} = m_{\operatorname{top}} (n, \varepsilon)$ such that the following holds.  For a positive integer $N$, let $\Delta^{N-1}$ be the $(N-1)$-dimensional simplex, with $N$ vertices.  For every $N \ge \varepsilon^{-1}(d+1)$ and every continuous map $f:\Delta^{N-1} \to \rr^n$, there is a set $\mathcal{T} \subset \rr^n$ of at most $m$ points such that the following holds.  For any set $A \subset \Delta^{N-1}$ of at least $\varepsilon N$ vertices,
	\[
	f[\langle A \rangle] \cap \mathcal{T} \neq \emptyset
	\]
	where $\langle A \rangle$ denotes the face of $\Delta^{N-1}$ generated by $A$.  Moreover, $m_{\operatorname{top}} = O(\varepsilon^{-n-1})$ where the implied constant of the $O$ notation depends on $n$.
\end{theorem}

\begin{proof}[Proof of Theorem \ref{theorem-weakepsilon-topological}]
We construct the set $\mathcal{T}$ inductively, by counting the number $K$ of faces $B$ of size $n+1$ such that $f[\langle B\rangle] \cap \mathcal{T} = \emptyset$.  Suppose there is a face $A$ with $|A| \ge \varepsilon N$ such that $f[\langle A \rangle ] \cap \mathcal{T} = \emptyset$.  Then, by the topological version of the first selection lemma applied to $A$, there must be a point $t \in \rr^n$ such that $t \in f[\langle B \rangle]$ for at least 
\[
c^*_n \binom{\varepsilon N}{n+1} \sim \varepsilon^{n+1}c^*_n \binom{N}{n+1}
\]
different subsets $B \subset A$ with $|B|=n+1$.  Thus, by adding the point $t$ to $\mathcal{T}$, we have reduced $K$ by at least $\varepsilon^{n+1}c^*_n \binom{N}{n+1}$.  This process cannot be repeated more than $(\varepsilon^{n+1}c^*_n)^{-1}$ times, so we obtain $m_{\operatorname{top}} \le (\varepsilon^{n+1}c^*_n)^{-1}$, as desired.

\end{proof}

Even though this result and its proof are natural with current methods, it seems that this generalisation has not yet been observed.  The proof of Theorem \ref{theorem-thin-hull-epsilon-net} is almost identical, but we include it for the sake of completeness.

\begin{proof}[Proof of Theorem \ref{theorem-thin-hull-epsilon-net}]
	We construct the set $\mathcal{T}$ inductively, by counting the number $K$ of pairs $\{x,y\} \in \binom{X}{2}$ such that $S(x,y) \cap \mathcal{T} = \emptyset$.  Suppose there is a subset $A$ with $|A|\ge \varepsilon |X|$ such that $\operatorname{thin}_S (A) \cap \mathcal{T} = \emptyset$.  Then, by Theorem \ref{theorem-positive-thin-shapes} applied to $A$, there must be a point $p \in \rr^n$ such that $t \in S(x,y)$ for at least
	\[
	\lambda (\varepsilon |X|)^2 \ge 2\lambda \varepsilon^2 \binom{|X|}{2}
	\]
	ordered pairs $(x,y) \in A \times A$.  Thus, by adding the point $p$ to $\mathcal{T}$, we have reduced $K$ by at least $\lambda \varepsilon^2 \binom{|X|}{2}$.  This process cannot be repeated more than $(\lambda \varepsilon^2)^{-1}$ times, giving the desired bound.
	
\end{proof}

The same proof method has been used to get other extensions of weak $\epsilon$-nets for convex sets, such as quantitative versions in \cite{soberon2015}.  When we apply Theorem \ref{theorem-thin-hull-epsilon-net} to $\alpha$-lenses (Example $1$ in the introduction), we obtain the following corollary.

\begin{corollary}\label{corollary-wide-angle-net}
	For any two real numbers $\alpha \in [0, \pi)$, $\varepsilon\in(0,1]$ and a positive integer $n$, there is an integer $m' = m'(n, \varepsilon, \alpha)$ such that the following holds.  For every finite set $S \subset \rr^n$, there is a set $\mathcal{T} \subset \rr^n$ of $m'$ points such that if $A \subset S$ is a subset of at least $\varepsilon |S|$ points, then there are $x,z \in A$ and $y\in \mathcal{T}$ such that
	\[
	\angle xyz > \alpha.
	\]
	Moreover, $m' = O(\varepsilon^{-2})$, where the implied constant of the $O$ notation depends on $n$ and  $\alpha$
\end{corollary}

In other words, there is a point of $\mathcal{T}$ that \textit{sees} some pair of points of every large subset of $S$ at a wide angle.  It is surprising that this notion of \textit{being close to $S$} allows for weak $\varepsilon$-nets of such small size.

This would be a counterpart to \cite[Theorem 4]{Barany:1987ef}.  In that result, B\'ar\'any and Lehel showed that for every angle $\alpha < \pi$ and every compact set $V \subset \rr^d$ has a subset $S$ of fixed size (depending only on $d$ and $\alpha$) such that every point in $X$ ``sees'' some pair of points of $S$ at an angle wider than $\alpha$.  In other words, if we denote by the $\alpha$-lens of $\{x,y\}$ the set of points $z$ such that $\angle xzy > \alpha$, it says that a fixed number of $\alpha$-lenses of $C$ cover the points in $C$.  In our result, we show that given the set $C$, a fixed number of points intersects ``most'' of the $\alpha$-lenses of $C$.

\bibliographystyle{amsplain}

\bibliography{references}

\noindent A. Magazinov, \\
\textsc{
R\'enyi Mathematical Institute, Budapest, Hungary and Yandex LLC, Moscow, Russia.
}\\[0.3cm]
\noindent P. Sober\'on \\
\textsc{
Department of Mathematics,  Northeastern University, Boston, MA 02115 USA
}\\[0.3cm]

\noindent \textit{E-mail addresses: }\texttt{magazinov-al@yandex.ru, p.soberonbravo@neu.edu}

\end{document}